\documentclass[psamsfonts]{amsart}

\usepackage{amssymb,amsfonts}
\usepackage[all,arc]{xy}
\usepackage{enumerate}
\usepackage{mathrsfs}
\usepackage{xcolor}
\usepackage[margin=1in]{geometry}
\usepackage{bbm}
\usepackage [english]{babel}
\usepackage [autostyle, english = american]{csquotes}
\MakeOuterQuote{"} 
\usepackage{tikz-cd} 
\usepackage{xpatch}
\usepackage{hyperref} 
\usepackage{amsrefs}     
\makeatletter
\renewcommand\th@plain{\slshape}
\xpatchcmd{\proof}{\itshape}{\slshape}{}{}
\makeatother

\newcommand{\absval}[1]{\left\lvert#1\right\rvert}

\newcommand{\ds}{\displaystyle}
\newcommand{\R}{\mathbb{R}}
\newcommand{\Z}{\mathbb{Z}}

\newcommand{\card}{\mathrm{card}}

\newcommand\ssm{\smallsetminus}

\newcommand{\ignore}[1]{}

\theoremstyle{plain}
\newtheorem{thm}{Theorem}[section]

\newtheorem{rmk}[thm]{Remark}
\newtheorem{prop}[thm]{Proposition}
\newtheorem{lem}[thm]{Lemma}

\theoremstyle{definition}
\newtheorem{defn}[thm]{Definition}

\newtheorem{notn}[thm]{Notation}

\newtheorem*{thme}{Theorem}
\newtheorem*{standing}{Standing Assumptions}   

\theoremstyle{remark}

\numberwithin{equation}{section}

\title[Higher-rank pointwise discrepancy bounds and logarithm laws for generic lattices]{Higher-rank pointwise discrepancy bounds and logarithm laws for generic lattices}  

\author{Seungki Kim}
\address{
\begin{itemize}
\item[] Seungki Kim    
\item[] Department of Mathematical Sciences  
\item[] University of Cincinnati   
\item[] $4199$ French Hall West
\item[] $2815$ Commons Way     
\item[] Cincinnati, OH $45221$\textendash$0025$  
\item[] USA
\item[] \href{seungki.math@gmail.com}{{\tt seungki.math@gmail.com}}
\item[] \href{https://orcid.org/0000-0003-1196-9370
}{{\tt https://orcid.org/0000-0003-1196-9370
}}
\end{itemize}}

\author{Mishel Skenderi}    
\address{
\begin{itemize}
\item[] Mishel Skenderi
\item[] Department of Mathematics
\item[] The University of Utah
\item[] $155$ South $1400$ East JWB $233$
\item[] Salt Lake City, UT $84112$\textendash$0090$ 
\item[] USA
\item[] \href{mailto:mskenderi@math.utah.edu}{{\tt mskenderi@math.utah.edu}}       
\item[] \href{https://orcid.org/0000-0001-8409-1613
}{{\tt https://orcid.org/0000-0001-8409-1613
}}   \end{itemize}}

\date{Thursday, July $7$, $2022$}

\begin{document}

\begin{abstract}
We prove a higher-rank analogue of a well-known result of W.\,M.~Schmidt concerning almost everywhere pointwise discrepancy bounds for lattices in Euclidean space (see Theorem 1 [Trans.~Amer.~Math.~Soc.~\textbf{95} (1960), 516--529]). We also establish volume estimates pertaining to higher minima of lattices and then use the work of Kleinbock--Margulis and Kelmer--Yu to prove dynamical Borel--Cantelli lemmata and logarithm laws for higher minima and various related functions.           
\end{abstract}  

\maketitle

\tableofcontents

\section{Introduction and Summary of Results}\label{PaperIntro}
\begin{notn}\label{notation}
We regard elements of Euclidean space of any dimension as row vectors. We let $n$ denote an arbitrary element of $\Z_{\geq 2}.$ We write $G := \operatorname{SL}_n(\R)$ and $\Gamma := \operatorname{SL}_n(\Z).$ We write $X := \Gamma \backslash G$; we then identify $X$ with the space of all covolume one full-rank lattices in $\R^n$ via the correspondence $\Gamma g \longleftrightarrow \Z^n g.$ We let $\mu_{G}$ denote the Haar measure on the unimodular group $G$ that is normalized in such a way that the covolume of $\Gamma$ in $G$ is equal to $1.$ We let $\mu_{X}$ denote the unique $G$-invariant Radon probability measure on $X.$ We note that the space of all full-rank lattices in $\R^n$ may be identified with $\ds \operatorname{GL}_n(\Z) \backslash \operatorname{GL}_n(\R) \cong X \times \R_{>0},$ where $\R_{>0}$ is equipped with its usual Haar measure $t^{-1} \ dt.$ We write $m$ to denote the Lebesgue measure on Euclidean space of any dimension; the dimension will be clear from the context. For any $r \in \R_{\geq 0},$ we let $B_r \subset \R^n$ denote the closed Euclidean ball that is centered at origin and has radius equal to $r.$ For any $t \in \R_{\geq 0},$ we let $\rho_t$ denote the indicator function of the closed Euclidean ball in $\R^n$ that is centered at the origin and has Lebesgue measure equal to $t.$ We let $\zeta$ denote the Riemann zeta function.

Let $\langle - , - \rangle_1$ denote the Euclidean inner product on $\R^n.$ Let $k$ be any integer with $1 \leq k \leq n.$ We define $\langle - , - \rangle_k$ to be the inner product on $\bigwedge^k\left(\R^n\right)$ given by $\ds \langle v_1 \wedge \dots \wedge v_k, w_1 \wedge \dots \wedge w_k \rangle_k := \det\left( \langle v_i, w_j \rangle_1 \right);$ we then let $\|\cdot\|_k$ denote the corresponding norm on $\bigwedge^k\left(\R^n\right)$. We shall often omit the subscripts of these norms when they are clear from the context. For any $\Lambda \in X,$ we define $X_k(\Lambda)$ to be the set of all subgroups of $\Lambda$ that have $\Z$-rank equal to $k.$ For any $\Lambda \in X$ and any $\Theta \in X_k(\Lambda),$ we define $\ds \det(\Theta) := \| w_1 \wedge \dots \wedge w_k \|,$ where $\ds \left(w_1, \dots , w_k\right)$ is any $\Z$-basis of $\Theta.$ (This definition is independent of the choice of $\Z$-basis.) We then define $\sigma_k : X \to \R_{>0}$ by $\ds \sigma_k(\Lambda) := \inf\left\lbrace \det(\Theta) : \Theta \in X_k(\Lambda) \right\rbrace$; note that this infimum is a minimum.

Throughout this paper, we use the Vinogradov notation $\ll$ and use $\asymp$ to denote that both "$\ll$" and "$\gg$" hold; we attach subscripts to the symbols $\ll$ and $\asymp$ to indicate the parameters, if any, on which the implicit constants depend.  \end{notn}

The purpose of this paper is to establish higher-dimensional generalizations of certain well-known results of a probabilistic flavor in the geometry of numbers. The first of these is a classical result of W.\,M.~Schmidt concerning almost everywhere pointwise discrepancy bounds.\footnote{It is a simple porism of Schmidt's proofs that the following theorem also holds when $\operatorname{GL}_n(\Z) \backslash \operatorname{GL}_n(\R)$ is replaced by $\operatorname{SL}_n(\Z) \backslash \operatorname{SL}_n(\R).$}

\begin{thme}\label{Schmidt}\cite[Theorem 1]{Metrical} Suppose $n \geq 3.$
\sl Let $\psi : \R_{\geq 0} \to \R_{>0}$ be any nondecreasing function with $\ds \int_0^{+\infty} (\psi(t))^{-1} \ dt <+\infty.$ Let $\Phi$ be a totally ordered collection of finite-volume Borel subsets of $\R^n$ with $\ds \sup\left\lbrace m(S) : S \in \Phi\right\rbrace = +\infty.$ Then for almost every $\Lambda \in \operatorname{GL}_n(\Z) \backslash \operatorname{GL}_n(\R)$, there exist constants $c_1(\Lambda), c_2(\Lambda) \in \R_{>0}$ such that for any $S \in \Phi$ with $m(S) \geq c_2(\Lambda)$, we have \[ \left| \frac{\det(\Lambda) \cdot \card(S \cap \Lambda)}{m(S)} - 1 \right| \leq c_1(\Lambda) \cdot \log(m(S)) \cdot (m(S))^{-1/2} \cdot (\psi(\log(m(S))))^{1/2}. \]
\end{thme}

The first aim of this paper is to generalize Schmidt's theorem to the case of counting $\ell$-tuples of lattice points. We start by defining the main object of our study, originally introduced by Siegel \cite{Siegel} and Rogers \cite{MeanRog, RogSet}.

%, let us first define the various so-called \textit{Siegel transforms} that we shall have occasion to use in the statement and proof of our theorem. 

\begin{defn}\label{Siegel}
Let $k$ be any integer with $1 \leq k \leq n-1.$ Let $\ds F : \left(\R^n\right)^k \to \R_{\geq 0}$ be a function. We define $\ds ^{k}\widehat{F}, \ ^{k}\widetilde{F} : X \to [0, +\infty]$ by  \begin{equation}\label{hatsummation}
^{k}\widehat{F}(\Lambda) := \sum_{(v_1, \dots , v_k) \in \left(\Lambda \ssm \{\mathbf{0} \}\right)^k} F(v_1, \dots , v_k) \end{equation} and 
\begin{equation}\label{tildesummation}
^{k}\widetilde{F}(\Lambda) := \sum F(v_1, \dots , v_k),     
\end{equation}
where the second sum ranges over all $(v_1, \dots , v_k) \in \Lambda^k$ for which $\ds \dim_{\R}\left( \mathrm{span}_{\R}\left( \{v_1, \dots , v_k \}\right)\right) = k.$ \end{defn}

Often in the literature, $^1\widehat{F}$ is referred to as the \textit{Siegel transform} of $F$. In this paper, we shall call any of its natural extensions, such as \eqref{hatsummation} and \eqref{tildesummation}, a Siegel transform as well.

\begin{rmk} \rm
Let $k$ be any integer with $1 \leq k \leq n-1.$ Let each of $A_1, \dots , A_k$ be a subset of $\R^n$; set $A := \prod_{j=1}^k A_j.$ For any $\Lambda \in X$, we have $^{k}\widehat{\mathbbm{1}_{A}}(\Lambda) = \card\left(\left(\Lambda \ssm \{\mathbf{0} \}\right)^k \cap A\right).$ An analogous statement holds for $^{k}\widetilde{\mathbbm{1}_{A}}.$
\end{rmk} \noindent We now state our first result. 

\begin{thm}\label{FirstMain} Suppose $n \geq 3.$ Let $\ell$ be any integer with $2 \leq 2\ell \leq n-1.$ Let $\psi : \R_{>0} \to \R_{>0}$ be a nondecreasing function for which $\ds \int_1^{+\infty} (\psi(t))^{-1} \, dt < +\infty.$ Let $\ds \{E_M\}_{M \in \R_{\geq 0}}$ be a collection of Borel subsets of $\R^n$ for which the following hold:~first, we have $m\left( E_M \right) = M$ for each $M \in \R_{\geq 0}$;~second, for any $M_1, M_2 \in \R$ with $0 \leq M_1 \leq M_2 < +\infty,$ we have $\ds E_{M_1} \subseteq E_{M_2}.$ Then for $\mu_X$-almost every $\Lambda \in X$ and for each $M \in \R_{\geq 1},$ we have \begin{equation}\label{boundfull} D^{(\ell)}\left(\Lambda, E_M^\ell\right) := \absval{\frac{^{\ell}\widehat{\mathbbm{1}_{E_M^\ell}}(\Lambda)}{m\left(E_M^\ell\right)} - 1 } \ll_{n, \ell, \Lambda} \ (\log M) M^{-1/2}(\psi(\log M))^{1/2} \end{equation} and \begin{equation}\label{boundindep}  D^{(\ell)}_{\rm indep}\left(\Lambda, E_M^\ell\right) := \absval{\frac{^{\ell}\widetilde{\mathbbm{1}_{E_M^\ell}}(\Lambda)}{m\left(E_M^\ell\right)} - 1 } \ll_{n, \ell, \Lambda} \ (\log M) M^{-1/2}(\psi(\log M))^{1/2}. \end{equation}  \end{thm}

\medskip

\noindent Let us mention that a much weaker version of Theorem \ref{FirstMain} was proved by the second-named author:~see \cite[Corollary 1.12 and Remark 2.22]{MS}. 
\begin{rmk}\label{strongremark} \rm
We shall give the proof of the bound \eqref{boundfull}; we shall omit the proof of the bound \eqref{boundindep}, which is almost identical to that of the bound \eqref{boundfull}. We also mention that our proof of Theorem \ref{FirstMain} can be easily modified to prove a more general result. One can replace the collection $\ds \{E_M\}_{M \in \R_{\geq 0}}$ by a collection $\ds \left\lbrace E_{j, M} : j \in \{1, \dots , \ell\} \ \text{and} \ M \in \R_{\geq 0} \right\rbrace$ of Borel subsets of $\R^n$ that satisfies the following conditions: 
\begin{itemize}
    \item There exists $\mathbf{c} = (c_1, \dots , c_\ell) \in \left(\R_{>0}\right)^\ell$ such that $\prod_{j=1}^\ell c_j = 1 $ and for each $j \in \{1, \dots , \ell \}$ and each $M \in \R_{\geq 0},$ we have $m\left(E_{j, M}\right) = c_jM.$ 
    \medskip     
    \item For each $j \in \{1, \dots , \ell \}$ and any $M_1, M_2 \in \R$ with $\ds 0 \leq M_1 \leq M_2 < +\infty,$ we have $\ds E_{j, M_1} \subseteq E_{j, M_2}.$
\end{itemize}
\medskip
\noindent Then the conclusion of Theorem \ref{FirstMain} holds when each instance of $\ds E_M^\ell$ in \eqref{boundfull} and \eqref{boundindep} is replaced by $\prod_{j=1}^\ell E_{j, M}$ and each instance of $\ll_{n, \ell, \Lambda}$ is replaced by $\ll_{n, \ell, \mathbf{c}, \Lambda}.$ 

\smallskip

One can define an analogue of the Siegel transform \eqref{tildesummation} in which the defining sum ranges over all primitive $k$-tuples $(v_1, \dots , v_k) \in \Lambda^k$, that is, those that can be completed to a $\Z$-basis of $\Lambda$. One can then prove a correspondingly analogous version of Theorem \ref{FirstMain}, together with the strengthening discussed earlier in Remark \ref{strongremark}. The proof proceeds in essentially the same fashion:~the only key difference is that one utilizes the expectation formula \cite[Satz 14]{TheorySchmidt} instead of Proposition \ref{indepSiegelRogers}. (It might be helpful to consult \cite[Theorem 1.10]{MS} and its proof and \cite[Remark 2.22]{MS}.) One can also consider other variations of the definition \eqref{tildesummation} and Theorem \ref{FirstMain}:~for instance, one can sum over all $k$-tuples $(v_1, \dots , v_k) \in \Lambda^k$ of $\R$-rank equal to $k$ and each of whose entries is a primitive vector in $\Lambda.$ It may be interesting to compare this particular analogue with \cite[Theorem 1]{AlievGruber}, which investigates the case $k=n$ but is quantitatively weaker than Theorem \ref{FirstMain}.
\end{rmk}  

\begin{rmk} \rm 
Suppose $n \geq 3,$ and let $\ell \in \Z_{\geq 1}$ be given. Let $A$ be a Borel subset of $\R^n$ with $0 < m(A) < +\infty.$ For every $p \in [1, n) \subset \R,$ we have $\ds ^{1}\widehat{\mathbbm{1}_A} \in L^p(X)$, as we shall see in the proof of Proposition \ref{polynomialbound}. For every $p \in \R_{\geq n},$ it follows from \cite[\S 2.2]{TheorySchmidt} that $\ds ^{1}\widehat{\mathbbm{1}_A} \notin L^p(X).$ If $2\ell \leq n-1$ (respectively, $2\ell > n-1$), then it follows that $\ds ^{\ell}\widehat{\mathbbm{1}_{A^\ell}} \in L^2(X)$ (respectively, $\ds ^{\ell}\widehat{\mathbbm{1}_{A^\ell}} \notin L^2(X)$). Since the proof of Theorem \ref{FirstMain} requires working with the second moment of $\ds ^{\ell}\widehat{\mathbbm{1}_{A^\ell}},$ this explains the assumption $2 \ell \leq n-1$ in the hypotheses of that theorem. Using the present methods, this assumption appears to be unavoidable. The authors explored the possibility of using fractional moments (strictly between the first and the second) to weaken the assumption $2 \ell \leq n-1$; however, such an approach quickly proved to be fruitless.   
\end{rmk}

\medskip

Our second result concerns dynamical Borel\textendash Cantelli lemmata and logarithm laws for flows on the space $X.$ The study of logarithm laws was initiated by D.\,P.~Sullivan in \cite{Sullivan}, in the context of the geodesic flow on the unit tangent bundle of certain finite-volume, noncompact hyperbolic manifolds. Sullivan seems to have coined the term "logarithm law" by analogy with Khintchine's law of the iterated logarithm \cite{Khintchine}:~see the discussion between Theorems 1 and 2 of \cite{Sullivan}. In \cite{KM}, Kleinbock\textendash Margulis proved a vast generalization of Sullivan's logarithm law as a consequence of their dynamical Borel\textendash Cantelli lemma \cite[Theorem 1.9]{KM}, which is a very general converse to the classical Borel\textendash Cantelli lemma. A special case of the Kleinbock\textendash Margulis logarithm law that is particularly important in our setting is the following theorem, in which $\beta_1(\Lambda)$ denotes the Euclidean length of any shortest nonzero vector of a given lattice $\Lambda.$

\begin{thme}\cite[Theorem 1.7 and Proposition 7.1]{KM}
\sl Let $\ds \{g_t\}_{t \in \R}$ be an unbounded $\R$-diagonalizable one-parameter subgroup of $G$. Then for almost every $\Lambda \in X,$ we have \[ \limsup_{t \to +\infty}  \frac{-\log\left(\beta_1\left( \Lambda g_t  \right)\right)}{\log t} = \frac{1}{n}. \]  
\end{thme}

Athreya\textendash Margulis in \cite{Log} then proved the unipotent analogue of the preceding theorem.  

\begin{thme}\cite[Theorem 2.1]{Log}
\sl Let $\ds \{g_t\}_{t \in \R}$ be an unbounded unipotent one-parameter subgroup of $G$. Let $\beta_1$ be as above. Then for almost every $\Lambda \in X,$ we have \[ \limsup_{t \to +\infty}  \frac{-\log\left(\beta_1\left( \Lambda g_t \right)\right)}{\log t} = \frac{1}{n}. \]  
\end{thme}

The approach of Kleinbock\textendash Margulis used a representation-theoretic result to ensure exponential decay of correlation coefficients for $\R$-diagonalizable flows (see \cite[Theorems 1.12 and 3.4 and Corollary 3.5]{KM}), together with the $n$-DL ("distance like") property of $-(\log \circ \beta_1)$ (see \cite[Definition 1.6]{KM} or Definition \ref{DLdefn}). The approach of Athreya\textendash Margulis, on the other hand, was comparatively elementary:~their main technical tool was a probabilistic analogue of the Minkowski Convex Body Theorem \cite[Theorem 2.2]{Log} that was proved for $n \geq 3$ using classical first and second moment formulae for the rank $1$ Siegel transform.\footnote{Athreya\textendash Margulis handled separately the more difficult case $n=2.$} The logarithm laws of Kleinbock\textendash Margulis and Athreya\textendash Margulis were then further generalized, at least in the spherical setting, in the work of Kelmer\textendash Yu:~see \cite[Theorem 1.1 and Corollary 1.2]{shrinking} for their logarithm laws and \cite[Theorem 1.3]{shrinking} for the representation-theoretic result on decay of correlation coefficients that undergirds their logarithm laws in the higher-rank case.         

In this paper, we apply the results of Kleinbock\textendash Margulis and Kelmer\textendash Yu to derive Borel\textendash Cantelli lemmata and logarithm laws for the higher Euclidean minima and the Koecher zeta functions; in doing so, we establish various volume estimates that are of independent interest. Let us now proceed by recalling the definition of the aforementioned DL ("distance-like") property. 

\medskip

\begin{defn}\label{DLdefn} \cite[Definition 1.6]{KM}
Assume that $X$ is equipped with the uniform structure induced by any norm on $\mathrm{Mat}_{n \times n}(\R),$ and let $\Delta : X \to \R$ be a uniformly continuous function. Given any $\alpha \in \R_{>0},$ we then say that $\Delta$ is $\alpha$-\textit{DL} if there exists $C \in \R_{>1}$ and there exists $M \in \R_{\geq 0}$ such that for each $z \in \R_{\geq M},$ we have 
\[ C^{-1} \exp(-\alpha z) \leq \mu_X\left(\left\lbrace \Lambda \in X : \Delta(\Lambda) \geq z \right\rbrace\right) \leq C\exp(-\alpha z). \]
\end{defn}

\medskip

We now introduce our functions of interest. 

\begin{defn}\label{succmin}
Let $k$ be any integer with $1 \leq k \leq n.$ Define $\ds \beta_k : X \to \R_{>0}$ by \[ \beta_k(\Lambda) := \inf\left\lbrace r \in \R_{>0} : \dim_{\R}\left( \mathrm{span}_{\R}\left( B_r \cap \Lambda \right)\right) \geq k \right\rbrace; \] note that this infimum is a minimum. For any $\Lambda \in X$, the quantity $\beta_k(\Lambda)$ is then the $k^{\rm th}$ \textit{successive minimum} of $\Lambda$ with respect to the Euclidean norm. 
\end{defn}

\smallskip 

\begin{defn}\label{Koecherzeta} Let $k$ be any integer with $1 \leq k \leq n.$ For any $\Lambda \in X$ and any $s \in \mathbb{C}$ with $\mathfrak{Re}(s) > n/2,$ define the absolutely convergent series \begin{equation}\label{Koeq}
\zeta_k(\Lambda, s) = \zeta_{n, k}(\Lambda, s) := \sum_{\Theta \in X_k(\Lambda)} \left(\det(\Theta)\right)^{-2s}. 
\end{equation} For each $\Lambda \in X,$ the analytic continuation of the Dirichlet series $\zeta_k(\Lambda, \cdot)$ is known as a \textit{Koecher zeta function}. These functions were introduced and studied by M.~Koecher in \cite{Koecher}. When $k=1$, these functions are usually known as the \textit{Epstein zeta functions}; they are named after P.~Epstein, who introduced them in \cite{Epstein} and studied them further in \cite{LaterEpstein}. When $k=n,$ the right-hand side of \eqref{Koeq} is equal to $\prod_{j=0}^{n-1} \zeta\left(2s-j\right).$ A good reference regarding the Koecher zeta functions is the book by A.\,A.~Terras \cite{Terras}. \end{defn} 

In \S \ref{VolumeEstimatesandLogarithmLaws}, we prove various volume estimates that allow us to establish that the functions in Definitions \ref{succmin} and \ref{Koecherzeta} satisfy the following DL conditions.  

\begin{prop}\label{prop:k-dl}
The following statements hold. 
\begin{itemize}
    \item[\rm (i)] For every integer $\ell$ with $1 \leq \ell \leq n-1,$ the function $-\left(\log \circ \beta_\ell\right)$ is $n\ell$-DL and the function \\
    $- \left( \log \circ \left( \prod_{j=1}^\ell \beta_j \right)\right) = - \sum_{j=1}^\ell\left( \log \circ  \beta_j \right)$ is $n$-DL.   
    \item[\rm (ii)] For every integer $\ell$ with $2 \leq \ell \leq n,$ the function $ \log \circ \beta_\ell$ is $n(n-\ell+1)$-DL and the function $\left( \log \circ \left( \prod_{j=\ell}^n \beta_j \right)\right) = \sum_{j=\ell}^n\left( \log \circ  \beta_j \right)$ is $n$-DL.
    \item[\rm (iii)] For every integer $\ell$ with $1 \leq \ell \leq n-1$ and every $s \in \left(n/2, +\infty\right) \subset \R,$ the function $\log \circ \zeta_\ell(\cdot, s)$ is $n/(2s)$-DL.
\end{itemize}
\end{prop}

We then apply \cite[Theorem 1.1]{shrinking} to obtain the following logarithm laws.        

\begin{thm}\label{loglaw}
Let $\ds \{g_t\}_{t \in \R}$ be an unbounded one-parameter subgroup of $G$. 

\begin{itemize}
\item[\rm (i)] For every integer $\ell$ with $1 \leq \ell \leq n-1$ and almost every $\Lambda \in X,$ we have \[ \limsup_{t \to +\infty}  \frac{-\log\left(\beta_\ell\left( \Lambda g_t \right)\right)}{\log t} = \frac{1}{n\ell} \] and \[ \limsup_{t \to +\infty}  \frac{-\log\left(\prod_{j=1}^\ell \beta_j\left( \Lambda g_t \right)\right)}{\log t} = \frac{1}{n}. \]

\medskip   

\item[\rm (ii)] For every integer $\ell$ with $2\leq \ell \leq n$ and almost every $\Lambda \in X,$ we have \[ \limsup_{t \to +\infty}  \frac{\log\left(\beta_\ell\left( \Lambda g_t \right)\right)}{\log t} = \frac{1}{n(n-\ell +1)} \]
and  \[ \limsup_{t \to +\infty}  \frac{\log\left(\prod_{j=\ell}^n \beta_j\left( \Lambda g_t \right)\right)}{\log t} = \frac{1}{n}. \]

\medskip

\item[\rm (iii)] For every integer $\ell$ with $1 \leq \ell \leq n-1$, every $s \in (n/2, +\infty) \subset \R$, and almost every $\Lambda \in X,$ we have \[ \limsup_{t \to +\infty}  \frac{\log\left(\zeta_\ell\left(\Lambda g_t, s\right)\right)}{\log t} = \frac{2s}{n}.  \]
\end{itemize} \end{thm}

\medskip  

In the event that the one-parameter subgroup of $G$ in Theorem \ref{loglaw} is $\R$-diagonalizable, one can apply \cite[Theorem 1.9]{KM} to obtain dynamical Borel\textendash Cantelli lemmata, and not only logarithm laws. We recall from \cite{KM} the following definition.

\begin{defn}\cite[Definition 1.5]{KM} 
Let $\ds \mathbf{g} := \{g_t\}_{t \in \R}$ be a one-parameter subgroup of $G.$ Let $\mathcal{D}$ be an arbitrary collection of Borel subsets of $X.$ We say that $\mathcal{D}$ is a \textit{Borel\textendash Cantelli family for} $\mathbf{g}$ if the following is true: for every mapping $D : \Z_{\geq 1} \to \mathcal{D}$, we have \begin{align*}
&\mu_X\left( \left\lbrace \Lambda \in X : \text{for each} \ M \in \Z_{\geq 1} \ \text{there exists} \ k = k_\Lambda \in \Z_{\geq M} \ \text{with} \ \Lambda g_k \in D(k) \right\rbrace\right) \\
&= \begin{cases} 
    0 & \text{if \  $\sum_{k=1}^{+\infty} \mu_X(D(k)) < +\infty$}, \\
    1 & \text{if \  $\sum_{k=1}^{+\infty} \mu_X(D(k)) = +\infty$}. 
  \end{cases}
\end{align*} \end{defn}

\smallskip       

\noindent The following result is then an immediate consequence of Proposition \ref{prop:k-dl} and \cite[Theorem 1.9]{KM}.

\begin{thm} \label{thm:b-c}
Let $\ds \mathbf{g} := \{g_t\}_{t \in \R}$ be an unbounded $\R$-diagonalizable one-parameter subgroup of $G$, and let $\Delta : X \to \R$ be any of the functions in Proposition \ref{prop:k-dl}. Then the collection $\mathcal{D} := \left\lbrace \left\lbrace \Lambda \in X : \Delta(\Lambda) \geq z \right\rbrace : z \in \R \right\rbrace$ is a Borel\textendash Cantelli family for $\mathbf{g}.$   
\end{thm}

%%%%%%%%%%%%%%%%
\section{Proof of Theorem \ref{FirstMain}}\label{ProofofFirstMain}
Our proof of Theorem \ref{FirstMain} is a modification of Schmidt's proof of \cite[Theorem 1]{Metrical}:~our proof relies crucially on the expectation formulae of Siegel \cite{Siegel} and Rogers \cite[(8) and Theorem 4]{MeanRog} for the Siegel transforms that were introduced in Definition \ref{Siegel}. Discrepancy results in the spirit of \cite[Theorem 1]{Metrical} and the so-called \textit{dyadic chaining} method that is used to prove them have a long history:~see, for instance, the pioneering papers \cite{Rademacher, Menchoff, GalKoksma}.   

\smallskip 

We begin by introducing two preliminary propositions; we shall use the first proposition in \S \ref{VolumeEstimatesandLogarithmLaws} as well.    

\begin{prop}\label{indepSiegelRogers}
Let $k$ be any integer with $1 \leq k \leq n-1.$ Let $F : \left(\R^n\right)^k \to \R_{\geq 0}$ be a Borel measurable function. Then \[ \int_X \,  ^{k}\widetilde{F} \, d\mu_X = \underbrace{\int_{\R^n} \cdots \int_{\R^n}}_{k \ \text{times}} F(x_1, \dots , x_k) \ dm(x_1) \cdots dm(x_k). \] \end{prop}
\begin{proof}
This result was initially stated without proof by Siegel (see \cite[page 347, 2)]{Siegel}) and was then proved by Rogers (see \cite[(8)]{MeanRog}). The special case $k=1$ is due to Siegel (\cite{Siegel}). 
\end{proof}

\begin{prop}\label{polynomialbound}
Let $k$ be any integer with $1 \leq k \leq n-1.$ Let $\mathbf{c} := (c_1, \dots , c_k) \in \left( \R_{>0} \right)^k$ be arbitrary. Let $V \in \R_{>0}$ be arbitrary. Let $A_1, \dots , A_k$ be any Borel subsets of $\R^n$ such that for each $j \in \{1, \dots , k\},$ we have $m\left(A_j\right) = c_j V.$ Then \[ \int_X \ ^{k}\widehat{\mathbbm{1}_{\prod_{j=1}^k A_j}} \, d\mu_X - \int_X \, ^{k}\widetilde{\mathbbm{1}_{\prod_{j=1}^k A_j}} \, d\mu_X = \int_X \, ^{k}\widehat{\mathbbm{1}_{\prod_{j=1}^k A_j}} \, d\mu_X - m\left(\prod_{j=1}^k A_j\right) \ll_{n, k, \mathbf{c}} \, V^{k-1}. \] 
\end{prop}
\begin{proof} The desired result is obvious when $n = 2$; we therefore assume $n \geq 3.$ It follows from Proposition \ref{indepSiegelRogers} that $\ds \int_X \, ^{k}\widetilde{\mathbbm{1}_{\prod_{j=1}^k A_j}} \, d\mu_X = m\left(\prod_{j=1}^k A_j\right).$ Since the desired result is clear for $k = 1,$ we now suppose $k \geq 2.$  For each $t \in \R_{>0},$ define $F_t : \left( \R^n \right)^k \to \R_{\geq 0}$ by $\ds F_t(x_1, \dots , x_k) := \prod_{j=1}^k \rho_{c_j t}(x_j),$ where the functions $\rho$ are as in Notation \ref{notation}. For each $t \in \R_{>0},$ we have $\ds \int_X \, ^{k}\widehat{F_t} \, d\mu_X < +\infty$ by \cite[Theorem 2]{MeanSchmidt}. The finitude of this integral can also be proved in a different manner. Set $c := \max\{c_1, \dots , c_k \}.$ It follows from \cite[Lemmata 3.1 and 3.10]{EMM} that\footnote{We remark that \cite[Lemma 3.1]{EMM} was proved by appealing to \cite[Lemma 2]{HeightSchmidt}.} for each $t \in \R_{>0}$ and each $p \in [1, n) \subset \R,$ we have $\ds ^{1}\widehat{\rho_{ct}} \in L^p(X).$ This clearly implies that for each $t \in \R_{>0},$ we have $\ds \int_X \, ^{k}\widehat{F_t} \, d\mu_X < +\infty.$

Let $r \in \{1, \dots , k -1\}$ be given, and let $\left(b_{i j}\right)_{1 \leq i \leq r, 1 \leq j \leq k}$ be an arbitrary element of $\operatorname{Mat}_{r \times k}(\R).$ It is easy to see that for each $t \in \R_{>0},$ we have \[ \underbrace{\int_{\R^n} \cdots \int_{\R^n}}_{r \ \textup{times}} \prod_{j=1}^k \rho_{c_j t}\left(\sum_{i=1}^r b_{i j} x_i \right) \, dm(x_1) \dots dm(x_r) = \left( \underbrace{\int_{\R^n} \cdots \int_{\R^n}}_{r \ \textup{times}} \prod_{j=1}^k \rho_{c_j}\left(\sum_{i=1}^r b_{i j} x_i \right) \, dm(x_1) \dots dm(x_r)\right) t^r. \] It now follows from Rogers's formula \cite[Theorem 4]{MeanRog} and the finitude of $\ds \int_X \, ^{k}\widehat{F_t} \, d\mu_X$ that there exists $(a_1, \dots , a_{k-1}) \in \left(\R_{>0}\right)^{k-1}$ such that for each $t \in \R_{>0},$ we have  \begin{equation}\label{Rogerspolynomial}
\int_X \, ^{k}\widehat{F_t} \, d\mu_X  = \left(\prod_{j=1}^k c_j\right) t^k + \sum_{j=1}^{k-1} a_j t^j.  
\end{equation} In light of \eqref{Rogerspolynomial} and the equation $\ds \int_X \, ^{k}\widetilde{\mathbbm{1}_{\prod_{j=1}^k A_j}} \, d\mu_X = m\left(\prod_{j=1}^k A_j\right)$, it remains only to establish 
\begin{equation}\label{BLLplusRogersineq}
\int_X \ ^{k}\widehat{\mathbbm{1}_{\prod_{j=1}^k A_j}} \, d\mu_X \leq  \int_X \, ^{k}\widehat{F_V} \, d\mu_X.
\end{equation} The Brascamp\textendash Lieb\textendash Luttinger rearrangement inequality \cite[Theorem 3.4]{BLL} yields  
\begin{align*}
&\underbrace{\int_{\R^n} \cdots \int_{\R^n}}_{r \ \textup{times}} \prod_{j=1}^k \mathbbm{1}_{A_j}\left(\sum_{i=1}^r b_{i j} x_i \right) \, dm(x_1) \dots dm(x_r) \\
\leq &\underbrace{\int_{\R^n} \cdots \int_{\R^n}}_{r \ \textup{times}} \prod_{j=1}^k \rho_{c_j V}\left(\sum_{i=1}^r b_{i j} x_i \right) \, dm(x_1) \dots dm(x_r).     
\end{align*} The desired inequality \eqref{BLLplusRogersineq} now follows from Rogers's formula \cite[Theorem 4]{MeanRog}. This completes the proof.   \end{proof}

\begin{rmk}\label{RemarkRogersErrors} \rm 
\begin{itemize}
\item[] 
\item[(i)] It was recently noticed that Rogers's proof of \cite[Theorem 4]{MeanRog} contains an error:~see \cite[\S 2]{KimSublattices} for details. Nevertheless, it is known that Rogers's formula \cite[Theorem 4]{MeanRog} is correct because Schmidt gave a different proof of it:~see \cite{GermanAgain, German}.  
\item[(ii)] In the $1957$ paper \cite{RogSing}, Rogers claimed to prove an inequality that is essentially equivalent to the Brascamp\textendash Lieb\textendash Luttinger rearrangement inequality \cite[Theorem 3.4]{BLL}. That being said, Rogers's proof is not entirely rigorous:~see \cite[Footnote 1]{CO} for a more detailed explanation. 
\end{itemize}
\end{rmk}

\smallskip  

\begin{standing} Now and throughout the remainder of \S \ref{ProofofFirstMain}, we suppose $n \geq 3$, we let $\ell$ be any integer with $2 \leq 2\ell \leq n-1$ (as in Theorem \ref{FirstMain}), and we let $\ds \{E_M\}_{M \in \R_{\geq 0}}$ be a collection of Borel subsets of $\R^n$ as in Theorem \ref{FirstMain}.  \end{standing}

\smallskip

For each $M \in \R_{\geq 0},$ we have $E_M^\ell \subset \left(\R^n\right)^\ell$; for each $M \in \R_{\geq 0},$ we define $R_M : X \to \R$ by $\ds R_M := \, ^{\ell}\widehat{\mathbbm{1}_{E_M^\ell}} - M^\ell \mathbbm{1}_X.$ For any $M_1, M_2 \in \R_{ \geq 0}$ with $M_1 \leq M_2,$ we define $\ds  _{M_1}R_{M_2}:= R_{M_2} - R_{M_1}.$ For each $T \in \Z_{\geq 1}$, we define
\[ K_T := \left\lbrace \left(u 2^t, (u+1)2^t \right) \in \Z^2 : (u, t) \in \left(\Z_{\geq 0}\right)^2 \ \text{and} \ (u+1)2^t \leq 2^T \right\rbrace. \]                 

\begin{lem} \label{lemma:oink}
For each $T \in \Z_{\geq 1}$ and each $(u, t) \in \left(\Z_{\geq 0}\right)^2$ for which $\left(u 2^t, (u+1)2^t \right) \in K_T,$ we have \[ \int_{X} \left(_{u 2^t}R_{(u+1)2^t}\right)^2 \ d\mu_X \ll_{n, \ell} 2^{T(2\ell-2)+t}. \]

For each $T \in \Z_{ \geq 1},$ we have \[ \sum_{(N, M) \in K_T} \int_{X} \left(_NR_M\right)^2 \ d\mu_X \ll_{n, \ell} T2^{T(2\ell-1)}.\] \end{lem} \noindent We defer the proof of this technical lemma to the end of \S \ref{ProofofFirstMain}. We shall also need the following lemma.       

\begin{lem}\label{inversepsi}
For each $T \in \Z_{\geq 1}$, there exists a Borel subset $C_T$ of $X$ that satisfies $\ds \mu_X\left(C_T\right) \ll_{n, \ell} \left(\psi(T)\right)^{-1}$ and for which the following holds:~for each $\Lambda \in (X \ssm C_T)$ and each $M \in \Z$ with $1 \leq M \leq 2^T$, we have  \[ R_M^2(\Lambda) \leq T^2 2^{T(2\ell-1)} \psi(T). \] \end{lem}
\begin{proof}
Let $T \in \Z_{\geq 1},$ and let $M$ be any integer with $1 \leq M \leq 2^T.$ For any $\Lambda \in X,$ we then have \begin{equation}\label{CST}
R_M(\Lambda) = \sum {}_NR_M(\Lambda),
\end{equation} where the sum ranges over at most $T$ pairs $(N,M) \in K_T.$ Now define $C_T$ to be the set of all $\Lambda \in X$ for which \[ \sum_{(N, M) \in K_T} {}_NR_M^2(\Lambda) > T2^{T(2\ell-1)} \psi(T). \] Markov's inequality and Lemma \ref{lemma:oink} then imply $\ds \mu_X\left(C_T\right) \ll_{n, \ell} \left(\psi(T)\right)^{-1}.$ Applying the Cauchy-Schwarz inequality to the sum on the right-hand side of \eqref{CST} implies that for each $\Lambda \in \left( X \ssm C_T \right),$ we have 
\[R_M^2(\Lambda) \leq T \cdot  T 2^{T(2\ell-1)} \psi(T) = T^2 2^{T(2\ell-1)} \psi(T).
\]\end{proof} We now prove Theorem \ref{FirstMain}.
\begin{proof}[Proof of Theorem \ref{FirstMain}]
Since $\ds \sum_{T=1}^{+\infty} \left(\psi(T)\right)^{-1} < +\infty,$ the Borel\textendash Cantelli lemma and Lemma \ref{inversepsi} imply the following: for $\mu_X$-almost every $\Lambda \in X$, there exists some $T_0(\Lambda) \in \Z_{\geq 1}$ such that for any $\ds T \in \Z_{\geq T_0(\Lambda)}$ and any $M \in \Z$ with $1 \leq M \leq 2^T,$ we have $\ds R_M^2(\Lambda) \leq T^2 2^{T(2\ell-1)} \psi(T).$ Since $\ds \R_{\geq 1} = \bigcup_{k=1}^{+\infty} \left[ 2^{k-1}, 2^k \right),$ we deduce the following: for $\mu_X$-almost every $\Lambda \in X,$ there exists $M_0(\Lambda) \in \Z_{\geq 1}$ such that for each $M \in \Z_{\geq M_0(\Lambda)},$ we have \[ 
R_M(\Lambda) \ll_{n, \ell, \Lambda} (\log M) M^{\ell - (1/2)}(\psi(\log M))^{1/2}. \] For each $M \in \Z_{\geq 1}$ and $\mu_X$-almost every $\Lambda \in X,$ we thus have \[ D^{(\ell)}\left(\Lambda, E_M\right) \ll_{n, \ell, \Lambda }(\log M) M^{-1/2}(\psi(\log M))^{1/2}. \] 

Now let $M' \in \R_{\geq 1}$ be given. Let $M''$ be an element of $\Z_{\geq 1}$ for which $M'' \leq M' \leq M'' + 1.$ For any $\Lambda \in X,$ we clearly have \[
    \absval{\card\left(\Lambda^\ell \cap E_{M'}^\ell\right) - \left( M' \right)^\ell} \leq \max\left\lbrace  \absval{\card\left(\Lambda^\ell \cap E_{M''}^\ell\right) - (M'' + 1)^\ell}, \absval{\card\left(\Lambda^\ell \cap E_{(M'' + 1)}^\ell\right) - \left( M'' \right)^\ell} \right\rbrace. \] For each $\Lambda \in X,$ it follows that $\ds D^{(\ell)}\left(\Lambda, E_{M'}\right)$ is less than or equal to \[
     \max\left\lbrace \frac{\left(M''\right)^\ell}{\left(M'\right)^\ell} \absval{ D^{(\ell)}\left(\Lambda, E_{M''}\right)  - \frac{\left(M'' + 1\right)^\ell - \left(M''\right)^\ell}{\left(M''\right)^\ell} }, \frac{\left(M''+1\right)^\ell}{\left(M'\right)^\ell}\absval{D^{(\ell)}\left(\Lambda, E_{\left(M''+1\right)}\right) + \frac{\left(M'' + 1\right)^\ell - \left(M''\right)^\ell}{\left(M''+1\right)^\ell}} \right\rbrace.   \] It now follows that for $\mu_X$-almost every $\Lambda \in X,$ we have \[ D^{(\ell)}\left(\Lambda, E_{M'}\right) \ll_{n, \ell, \Lambda }(\log M') (M')^{-1/2}(\psi(\log M'))^{1/2}. \]       
\end{proof}

\begin{proof}[Proof of Lemma \ref{lemma:oink}]
The second statement is an immediate consequence of the first one: if the first statement is true, then for each $T \in \Z_{\geq 1}$, we have  \[ \sum_{(N, M) \in K_T} \int_{X} \left(_NR_M\right)^2 \ d\mu_X \ll_{n, \ell} \sum_{t=0}^T \left( 2^{T-t} \cdot  2^{T(2\ell-2)+t} \right) = (T+1)2^{T(2\ell-1)}. \]

We now prove the first statement. Let $T \in \Z_{\geq 1}$ be given. Fix any $(u, t) \in \left(\Z_{\geq 0}\right)^2$ for which $\left(u 2^t, (u+1)2^t \right) \in K_T.$ Set $N := u 2^t$ and $M := (u+1)2^t.$ Then $M-N = 2^t.$ Let $\mathcal{J}(M, N)$ denote the set of all \[ \left( (N_1, M_1), \dots , (N_\ell, M_\ell) \right) \in K_T^\ell \] for which the following hold:~for each $j \in \{1, \dots , \ell \},$ we have $M_j - N_j = 2^t$ and there exists some $i \in \{1, \dots , \ell\}$ for which $M_i = M.$ Then $\card\left( \mathcal{J}(M, N) \right) = (u+1)^\ell - u^\ell \ll_\ell u^{\ell-1}$ and \[ E_M^\ell \ssm E_N^\ell = \bigsqcup_{\left( (N_1, M_1), \dots , (N_\ell, M_\ell) \right) \in \mathcal{J}(M, N)} \ \prod_{j=1}^\ell \left( E_{M_j} \ssm E_{N_j} \right). \] Notice that the above union is disjoint and that $\ds M^\ell - N^\ell = \left[(u+1)^\ell - u^\ell\right]2^{t\ell} \ll_{\ell} \, u^{\ell-1} 2^{t\ell}.$ We then have  \begin{equation*}
    _NR_M = \sum_{\left( (N_1, M_1), \dots , (N_\ell, M_\ell) \right) \in \mathcal{J}(M, N)} \left( \  ^{\ell}\widehat{f} - 2^{t\ell}\, \mathbbm{1}_X\right), \end{equation*} 
where \begin{equation}\label{simplenostar}
f := \mathbbm{1}_{\prod_{j=1}^\ell \left( E_{M_j} \ssm E_{N_j} \right)}.
\end{equation} It then follows that
\begin{equation}\label{hugesum}
\left( _NR_M \right)^2 = \sum \Big[ \  ^{\ell}\widehat{f} \cdot \  ^{\ell}\widehat{f_*} - 2^{t\ell}  \left( \  ^{\ell}\widehat{f} + \ ^{\ell}\widehat{f_*} \right) +  2^{2t\ell} \, \mathbbm{1}_X \Big] , \end{equation}
where the sum ranges over all $\left( (N_1, M_1), \dots , (N_\ell, M_\ell), (N_{1, *}, M_{1, *}) , \dots , (N_{\ell, *}, M_{\ell, *})\right) \in \left(\mathcal{J}(M, N)\right)^2,$ $f$ is as in \eqref{simplenostar}, and \begin{equation}\label{simplestar}
 f_* := \mathbbm{1}_{\prod_{j=1}^\ell \left( E_{{M_{j, *}}} \ssm E_{N_{j, *}} \right)}. 
\end{equation} Note that $\card\left( \left(\mathcal{J}(M, N) \right)^2 \right) = \left[(u+1)^\ell - u^\ell\right]^2 \ll_\ell u^{2\ell-2}.$
To prove the desired result, it suffices to show that the integral over $X$ with respect to $\mu_X$ of each summand on the right-hand side of \eqref{hugesum} is $\ll_{n, \ell} \, 2^{t(2\ell-1)}$, as this would then imply  
\begin{equation*}
    \int_X \left(_NR_M\right)^2 \ d\mu_X \ll_{n, \ell} \, u^{2\ell-2} \cdot 2^{t(2\ell-1)} \ll_{n, \ell} \, \left( 2^{T-t}\right)^{2\ell-2} \cdot 2^{t(2\ell-1)} = 2^{T(2\ell-2) + t}.       
\end{equation*}

\smallskip

Now let $ \left( (N_1, M_1), \dots , (N_\ell, M_\ell), (N_{1, *}, M_{1, *}) , \dots , (N_{\ell, *}, M_{\ell, *})\right) \in \left(\mathcal{J}(M, N)\right)^2$ be given; let $f$ and $f_*$ be as in \eqref{simplenostar} and \eqref{simplestar}, respectively. Note that $\ds \  ^{\ell}\widehat{f} \cdot \  ^{\ell}\widehat{f_*}$ is equal to the $^{2\ell}\widehat{\left(\cdot\right)}$ transform of the indicator function of $\prod_{i=1}^\ell \left( E_{M_i} \ssm E_{N_i} \right) \times \prod_{j=1}^\ell \left( E_{{M_{j, *}}} \ssm E_{N_{j, *}} \right),$ which is a subset of $\left( \R^n \right)^{2\ell}.$ It now follows from Proposition \ref{polynomialbound} that we have

\begin{equation*}
    \int_X \left(  \  ^{\ell}\widehat{f} \cdot \  ^{\ell}\widehat{f_*} \right) \ d\mu_X - 2^{t(2\ell)} \ll_{n, \ell} \, \left(2^t\right)^{2\ell-1} = 2^{t(2\ell-1)} 
\end{equation*} and \begin{equation*}
    \int_X \left(  \  ^{\ell}\widehat{f} + \  ^{\ell}\widehat{f_*} \right) \ d\mu_X - \left( 2 \cdot 2^{t\ell} \right) \ll_{n, \ell} \, \left( 2 \cdot 2^{t(\ell-1)}\right) \ll_{n, \ell} \, 2^{t(\ell-1)}. 
\end{equation*}

\smallskip

\noindent Since $\ds 2^{t(2\ell)} - 2^{t\ell} \left( 2 \cdot 2^{t\ell} \right) + 2^{2t\ell} = 0 $ \ and \ $\ds 2^{t(2\ell-1)} + \left( 2^{t\ell} \cdot 2^{t(\ell-1)} \right) = 2^{t(2\ell-1)} + 2^{t(2\ell-1)} \ll 2^{t(2\ell-1)}$, it follows that the integral over $X$ with respect to $\mu_X$ of each summand on the right-hand side of \eqref{hugesum} is $\ll_{n, \ell} \, 2^{t(2\ell-1)}.$  \end{proof}

\section{Volume Estimates and Logarithm Laws}\label{VolumeEstimatesandLogarithmLaws} 

We begin by stating the volume estimates to which we alluded in \S \ref{PaperIntro}. 

\begin{prop}\label{volestimate}
Let $\ell$ be any integer with $1 \leq \ell \leq n-1.$ For any $r \in \R_{>0},$ we have 
\[ \mu_{X}\left( \left\lbrace \Lambda \in X : \beta_\ell(\Lambda) \leq r \right\rbrace \right) \ll_{n, \ell} \, r^{n\ell} \] and \[ \mu_{X}\left( \left\lbrace \Lambda \in X : \left(\prod_{j=1}^\ell \beta_j(\Lambda) \right) \leq r \right\rbrace \right) \ll_{n, \ell} \, r^{n}.\]

\medskip        

\noindent Furthermore, there exists $ \varepsilon_{0, n, \ell} \in \R_{>0}$ such that for any $\varepsilon \in \left(0, \varepsilon_{0, n, \ell} \right) \subset \R,$ we have \[ \varepsilon^{n\ell} \ll_{n, \ell} \, \mu_{X}\left( \left\lbrace \Lambda \in X : \beta_\ell(\Lambda) \leq \varepsilon \right\rbrace \right)  \] and \[ \varepsilon^{n} \ll_{n, \ell} \, \mu_{X}\left( \left\lbrace \Lambda \in X : \left(\prod_{j=1}^\ell \beta_j(\Lambda) \right) \leq \varepsilon \right\rbrace \right).  \]
\end{prop}

\noindent Let us mention that the $\ell = 1$ case of the first statement of Proposition \ref{volestimate} has already been proved by Kleinbock\textendash Margulis with a sharper lower bound than the one stated here:~see \cite[Proposition 7.1]{KM}. Let us also mention that the upper bounds in this proposition are well-known, as they are immediate consequences of well-known expectation formulae in the geometry of numbers.     

\begin{prop}\label{dualvolestimate}
Let $\ell$ be any integer with $2 \leq \ell \leq n.$ For any $R \in \R_{>0},$ we have 
\[ \mu_{X}\left( \left\lbrace \Lambda \in X : \beta_\ell(\Lambda) \geq R \right\rbrace \right) \ll_{n, \ell} \, R^{-n(n-\ell+1)} \] and \[ \mu_{X}\left( \left\lbrace \Lambda \in X : \left(\prod_{j=\ell}^n \beta_j(\Lambda) \right) \geq R \right\rbrace \right) \ll_{n, \ell} \, R^{-n}.\]

\medskip        

\noindent Furthermore, there exists $M_{n, \ell} \in \R_{>0}$ such that for any $M \in \left(M_{n, \ell}, +\infty \right) \subset \R,$ we have \[ M^{-n(n-\ell+1)} \, \ll_{n, \ell} \, \mu_{X}\left( \left\lbrace \Lambda \in X : \beta_\ell(\Lambda) \geq M \right\rbrace \right)  \] and \[ M^{-n} \, \ll_{n, \ell} \, \mu_{X}\left( \left\lbrace \Lambda \in X : \left(\prod_{j=\ell}^n \beta_j(\Lambda) \right) \geq M \right\rbrace \right).  \]
\end{prop}

\begin{prop}\label{Koechervol}
Let $\ell$ be any integer with $1 \leq \ell \leq n-1,$ and let $s \in \left(n/2, +\infty\right) \subset \R$ be given. For any $R \in \R_{>0},$ we have \[ \mu_X\left(\left\lbrace \Lambda \in X : \zeta_\ell(\Lambda, s) \geq R \right\rbrace\right) \ll_{n, \ell, s} \, R^{-n/(2s)}.  \] There exists $M_{n, s} \in \R_{>0}$ such that for any $M \in \left(M_{n, s}, +\infty \right) \subset \R,$ we have \[ M^{-n/(2s)} \ll_{n, \ell} \, \mu_X\left(\left\lbrace \Lambda \in X : \zeta_\ell(\Lambda, s)  \geq M \right\rbrace\right).  \]  \end{prop}

Before we prove the preceding propositions, let us recall the following elementary generalization of Minkowski's Second Theorem (Minkowski's theorem on successive minima). (Recall from Notation \ref{notation} the relevant notation.)

\begin{lem}\label{GenSecondMink} For any integer $k$ with $1 \leq k \leq n$ and any $\Lambda \in X,$ we have $\ds \sigma_k(\Lambda) \asymp_{n, k} \prod_{j=1}^k \beta_j(\Lambda).$ \end{lem} \noindent  We now begin to prove the volume estimates.     

\begin{proof}[Proof of the upper bounds in Proposition \ref{volestimate}]
Let $t \in \R_{>0}.$ Let $W$ denote the closed Euclidean ball in $\R^n$ of volume $t$ with center at the origin. Define $\ds F : \left(\R^n\right)^\ell \to \R_{\geq 0}$ by $ F(x_1, \dots , x_\ell) := \prod_{i=1}^\ell \rho_t(x_i).$ Proposition \ref{indepSiegelRogers} then yields $\ds \int_X \,  ^{\ell}\widetilde{F} \, \mu_X = t^\ell.$ We then have \[
    \mu_{X}\left( \left\lbrace \Lambda \in X : \dim_\R\left({\rm span}_\R\left( \Lambda \cap W \right) \right) \geq \ell \right\rbrace \right) = \mu_{X}\left( \left\lbrace \Lambda \in X : \ ^{\ell}\widetilde{F}(\Lambda) \geq \ell ! \, 2^\ell \right\rbrace \right) \leq \left(\ell ! \, 2^\ell \right)^{-1} \cdot t^\ell.   
\] We may thus take the implicit constant in the first upper bound to be $\ds \left(\ell ! \, 2^\ell \right)^{-1} \cdot V_n^\ell,$ where $V_n$ denotes the volume of the closed Euclidean ball in $\R^n$ of radius $1$ with center at the origin. 

\smallskip

We now establish the second upper bound. Let $H \in \R_{>0}.$ Define $\varphi : X \to \R_{\geq 0}$ by \[\varphi(\Lambda) := \card\left(\left\lbrace \Theta \in X_\ell(\Lambda) : \det(\Theta) \leq H \right\rbrace\right).\] By \cite[Lemma 5]{Thunder}, we have\footnote{\label{adelicclassical}The result \cite[Lemma 5]{Thunder} is formulated in the ad\`{e}lic language; one can equally well appeal to \cite[Theorem 3]{KimSublattices}, which is formulated in the classical language.} $\ds \int_X \varphi \ d\mu_X \ll_{n, \ell} H^n.$ Therefore, \[ \mu_{X}\left( \left\lbrace \Lambda \in X : \sigma_\ell(\Lambda) \leq H \right\rbrace \right) = \mu_{X}\left( \left\lbrace \Lambda \in X : \varphi(\Lambda) \geq 1 \right\rbrace \right) \leq \int_X \varphi \ d\mu_X \ll_{n, \ell} H^n. \] An application of Lemma \ref{GenSecondMink} then completes the proof.   \end{proof}

\smallskip 

We now proceed by recalling the Iwasawa decomposition of $G$ and related matters concerning the reduction theory of $G.$ A good reference for this material is the book by A.~Borel \cite{ArithmeticBorel}.   

\begin{notn}
Define $K := \operatorname{SO}(n).$ Let $A$ denote the set of all diagonal matrices in $G.$ Let $U$ denote the set of all elements of $G$ that are upper-triangular and have each diagonal entry equal to $1.$ Each of $K,$ $A$, and $U$ is a closed subgroup of $G$; the Iwasawa decomposition of $G$ then asserts that the mapping $K \times A \times U \to G$ given by $\ds (k, a, u) \mapsto kau$ is a $\mathscr{C}^\infty$-diffeomorphism. 

\medskip

We denote by $\mu_{K}$ the Haar probability measure on the compact group $K.$ 

\medskip

We denote by $\mu_{A}$ the Haar measure on $A$ that is given by \[ d\mu_{A}(a) = d\mu_{A}\left(\operatorname{diag}\left( a_1, \dots , a_n \right)\right) = \prod_{i=1}^{n-1} a_i^{-1} \ d a_i. \] 

\medskip

Notice that $U$ may be identified with $\R^{n(n-1)/2}$; we then define the Haar measure $\mu_{U}$ on $U_n$ by \[ d\mu_{U}(u) := \prod_{1 \leq i < j \leq n} d u_{ij}. \] 

\medskip  

In the coordinates afforded by the Iwasawa decomposition, the Haar measure $\mu_{G}$ on $G$ is then given by \[ d\mu_{G}(g) = d\mu_{G}(kau) = \omega_n \, \prod_{1 \leq i < j \leq n} \frac{a_i}{a_j} \, d\mu_{A}(a) \, d\mu_{U}(u) \, d\mu_{K}(k), \] where $\omega_n$ is a constant that depends only on $n.$ 

\medskip

Now let $A_{2}$ denote the set of all $\ds a = \operatorname{diag}\left( a_1, \dots , a_n \right) \in A$ such that for each $i \in \Z$ with $1 \leq i \leq n-1,$ we have $\ds \frac{a_i}{a_{i+1}} \leq 2.$ 

\medskip

Let $U_{1}$ denote the set of all matrices in $U$ each of whose entries has absolute value less than or equal to $1.$ We evidently have $\ds \mu_{U}\left(U_{1}\right) = 2^{n(n-1)/2}.$ 

\medskip         

Now let $\mathfrak{S}$ be the so-called \textit{Siegel set} given by $\mathfrak{S} := K \, A_{2} \, U_{1} \subset G.$ Let $\ds p_X : G \to X$ denote the quotient map. It is well known that $\ds \mu_{G}\left(\mathfrak{S}\right)$ is finite and that $\ds p(\mathfrak{S}) = X.$ In fact, much more is true. The set $\mathfrak{S}$ contains a fundamental domain for the left action of $\Gamma$ on $G$; moreover, $\mathfrak{S}$ is contained in the union of finitely many fundamental domains for the left action of $\Gamma$ on $G.$ These properties of $\mathfrak{S}$ are very important; they ensure that there exists some $\lambda_{n, \mathfrak{S}} \in \Z_{\geq 2}$ such that for any Borel subset $E$ of $X,$ we have \begin{equation}\label{RNderivatives}
   \mu_X(E) \leq \mu_G\left( \mathfrak{S} \cap p^{-1}(E) \right) \leq \lambda_{n, \mathfrak{S}} \cdot \mu_X(E).  \end{equation} Thus, while $\mathfrak{S}$ is not a fundamental domain, it is almost as good as one for many measure-theoretic purposes. Finally, let us mention the property of $\mathfrak{S}$ that will be of greatest importance to us. For any integer $\ell$ with $1 \leq \ell \leq n,$ define $\pi_{\ell} : G \to \R_{>0}$ to be the map given by \[ \begin{tikzcd}
G \arrow{r}  & K \times A \times U \arrow{r}  & A \arrow{r}  & \R_{>0},
\end{tikzcd} \] where the first map is given by the Iwasawa decomposition, the second map is the obvious projection, and the third map is the projection onto the $(\ell, \ell)$ diagonal entry. We then have the following important result, which we record as a lemma. \end{notn} 

\begin{lem}\label{diagonalminima}
For any $g \in \mathfrak{S}$ and any integer $\ell$ with $1 \leq \ell \leq n,$ we have $\ds  \pi_{\ell}(g) \asymp_n \beta_\ell\left(\Gamma g\right). $   
\end{lem}
\begin{proof}
For a proof, see the Remark that immediately follows the proof of \cite[Proposition 1.12]{LLLReduction}. \end{proof}

\begin{defn}
For any $\varepsilon \in (0, 1) \subset \R$ and any integer $\ell$ with $1 \leq \ell \leq n-1,$ define $\ds \mathfrak{S}_{\ell, \varepsilon} := \mathfrak{S} \cap \pi_\ell^{-1}\left((0, \varepsilon] \right).$ 
\end{defn} We then have the following result, which will easily yield the lower bounds in Proposition \ref{volestimate}. 

\begin{lem}\label{lowerestimate}
For every $\varepsilon \in (0, 1) \subset \R$ and any integer $\ell$ with $1 \leq \ell \leq n-1,$ we have $\ds  \varepsilon^{n\ell} \ll_{n, \ell} \, \mu_{G}\left(\mathfrak{S}_{\ell, \varepsilon}\right). $   
\end{lem}
\begin{proof}
Let $\varepsilon \in (0, 1) \subset \R$ be given. Let $\ell$ be any integer with $1 \leq \ell \leq n-1.$ Let $A_{\ell, \varepsilon}$ denote the image of $\mathfrak{S}_{\ell, \varepsilon}$ under the map $G \to A$ given by composing the Iwasawa decomposition map $G \to K \times A \times U$ with the projection map $K \times A \times U \to A.$ Then \[ \mu_{G}\left(\mathfrak{S}_{\ell, \varepsilon}\right) = \omega_n \cdot 2^{[n(n-1)]/2} \cdot \int_{A_{\ell, \varepsilon}} \ \prod_{1 \leq i < j \leq n} \frac{a_i}{a_j} \, d\mu_{A}(a).  \] In order to estimate this integral, we perform a change of coordinates. 

\medskip         

For each $i \in \{1, \dots , n - 1\},$ define $\ds b_i := \frac{a_i}{a_{i+1}}.$ Then $\ds b := \left(b_1, \dots , b_{n-1}\right)$ constitutes a coordinate system on $A.$ One easily verifies that the Jacobian determinant of the change-of-coordinates from $a$ to $b$ is $\ds \frac{1}{2a_1}.$ Another easy calculation yields $\ds \prod_{1 \leq i < j \leq n} \frac{a_i}{a_j} = \prod_{i=1}^{n-1} b_i^{i(n-i)}.$ It follows 

\begin{align*}
    \int_{A_{\ell, \varepsilon}} \ \prod_{1 \leq i < j \leq n} \frac{a_i}{a_j} \, d\mu_{A_n}(a) = \int_{A_{\ell, \varepsilon}} \prod_{i=1}^{n-1} b_i^{i(n-i)} \, d\mu_{A_n}(a) &= \int_{A_{\ell, \varepsilon}} \  \prod_{i=1}^{n-1} \left( b_i^{i(n-i)}a_i^{-1} \, d a_i \right) \\
    &= \int_{A_{\ell, \varepsilon}} \  \frac{1}{2a_1}\prod_{i=1}^{n-1} \left( b_i^{i(n-i)}a_i^{-1} \, d b_i \right) \\
    &= \int_{A_{\ell, \varepsilon}} \  \frac{1}{2a_1}\prod_{i=1}^{n-1} \left( b_i^{i(n-i)}b_i^{-1} a_{i+1}^{-1} \, d b_i \right) \\
    &= \int_{A_{\ell, \varepsilon}} \  \frac{1}{2} \left( \prod_{j=1}^n \frac{1}{a_j}  \right)\prod_{i=1}^{n-1} \left( b_i^{-1+i(n-i)} \, d b_i \right) \\
    &= \frac{1}{2} \ \int_{A_{\ell, \varepsilon}}   \prod_{i=1}^{n-1}  b_i^{-1+i(n-i)} \, d b_i. \\
\end{align*}  We now wish to express the domain of integration $A_{\ell, \varepsilon}$ in terms of the coordinates $\ds b = \left( b_1, \dots , b_{n-1} \right).$ For the sake of notational convenience, define $b_0 := 1$ and $b_n := 1.$ It is easy to see that for each $j \in \{1, \dots , n \},$ we have $\ds a_1 = a_{j} \, \prod_{i=0}^{j-1} b_i.$ It follows 

\[ a_1^n = \prod_{j=1}^{n} a_{j} \, \prod_{i=0}^{j-1} b_i = (a_1 \cdots a_n) \cdot \prod_{i=1}^{n-1} b_i^{(n-i)} = \prod_{i=1}^{n-1} b_i^{(n-i)} = \left( \prod_{i=0}^{1-1} b_i^{-i}\right) \left( \prod_{i=1}^n b_i^{(n-i)}\right). \] 

\medskip   

For each $i \in \{1, \dots , n-1 \},$ we have $\ds a_{i+1} = \frac{a_i}{b_i}$ and thus have $\ds a_{i+1}^n = \frac{a_i^n}{b_i^n}.$

\medskip       

It is now easy to see that for each $j \in \{1, \dots , n \},$ we have \[ a_j^n = \left( \prod_{i=0}^{j-1} b_i^{-i}\right) \left( \prod_{i=j}^n b_i^{(n-i)}\right). \] 

We consider three distinct cases. 

\medskip 

Suppose first $\ds 1 < \ell < n-1.$ Let $A_{\ell, \varepsilon, *}$ denote the set of all $\ds a = \left(a_1, \dots , a_n\right) \in A$ for which the following hold:
\begin{itemize}
    \item[(i)] For each $i \in \Z$ with $1 \leq i \leq \ell-1,$ we have $1 \leq b_i \leq 2.$  
    \item[(ii)] For each $i \in \Z$ with $\ell + 1 \leq i  \leq n-1,$ we have $b_i \leq 1.$ 
    \item[(iii)] We have $\ds b_\ell \leq \varepsilon^{\frac{n}{(n-\ell)}}.$
\end{itemize}

\medskip  

We clearly have $\ds A_{\ell, \varepsilon, *} \subseteq A_{\ell, \varepsilon}.$ It follows 

\begin{align*}
    \int_{A_{\ell, \varepsilon}} \  \prod_{i=1}^{n-1} b_i^{-1+i(n-i)} \ d b_i &\geq \int_{A_{\ell, \varepsilon, *}} \  \prod_{i=1}^{n-1} b_i^{-1+i(n-i)} \ d b_i \\
    &= \left(\prod_{i=1}^{\ell-1} \left[ \frac{b_i^{i(n-i)}}{i(n-i)}\right]_1^2\right) \cdot \left(\prod_{i=\ell+1}^{n-1} \left[ \frac{b_i^{i(n-i)}}{i(n-i)}\right]_0^1 \right) \cdot \left[ \frac{b_\ell^{\ell(n-\ell)}}{\ell(n-\ell)}\right]_0^{\varepsilon^{\frac{n}{(n-\ell)}}} \\
    &\asymp_{n, \ell} \, \varepsilon^{\frac{n}{(n-\ell)} \cdot \ell(n-\ell)} \\
    &= \varepsilon^{n\ell}.  
\end{align*} 

This yields the desired result in the case $1 < \ell < n-1.$ 

\medskip  

Suppose now $\ell = 1.$ Let $\ds A_{1, \varepsilon, *}$ denote the set of all $a = (a_1, \dots , a_n) \in A$ such that $\ds b_1 \leq \varepsilon^{\frac{n}{(n-1)}}$ and for each $i \in \Z$ with $2 \leq i \leq n-1,$ we have $\ds b_i \leq 1.$ We have $\ds A_{1, \varepsilon, *} \subseteq A_{1, \varepsilon},$ whence 

\begin{align*}
    \int_{A_{1, \varepsilon}} \  \prod_{i=1}^{n-1} b_i^{-1+i(n-i)} \ d b_i \geq \int_{A_{1, \varepsilon, *}} \  \prod_{i=1}^{n-1} b_i^{-1+i(n-i)} \ d b_i &= \left[ \frac{b_1^{1(n-1)}}{1(n-1)}\right]_0^{\varepsilon^{\frac{n}{(n-1)}}}  \cdot \left(\prod_{i=2}^{n-1} \left[ \frac{b_i^{i(n-i)}}{i(n-i)}\right]_0^1 \right) \\
    &\asymp_n \, \varepsilon^{\frac{n}{(n-1)} \cdot 1(n-1)} \\
    &= \varepsilon^{n}.  
\end{align*} 

This implies the desired result in the case $\ell = 1.$ \\

Finally, suppose $\ell = n-1.$  Let $\ds A_{n-1, \varepsilon, *}$ denote the set of all $a = (a_1, \dots , a_n) \in A$ such that $\ds b_{n-1} \leq \varepsilon^n $ and for each $i \in \Z$ with $1 \leq i \leq n-2$, we have $\ds 1 \leq b_i \leq 2.$ Since $\ds A_{n-1, \varepsilon, *} \subseteq A_{n-1, \varepsilon},$ we have 

\begin{align*}
    \int_{A_{n-1, \varepsilon}} \  \prod_{i=1}^{n-1} b_i^{-1+i(n-i)} \ d b_i \geq \int_{A_{n-1, \varepsilon, *}} \  \prod_{i=1}^{n-1} b_i^{-1+i(n-i)} \ d b_i &= \left(\prod_{i=1}^{n-2} \left[ \frac{b_i^{i(n-i)}}{i(n-i)}\right]_1^2\right)  \cdot \left[ \frac{b_{n-1}^{n-1}}{(n-1)}\right]_0^{\varepsilon^{n}} \\
    &\asymp_n \, \varepsilon^{n(n-1)}. 
\end{align*} This completes the proof.  \end{proof}

We may now complete the proof of Proposition \ref{volestimate}.

\begin{proof}[Proof of the lower bounds in Proposition \ref{volestimate}]
The first lower bound is an immediate consequence of Lemmata \ref{diagonalminima} and \ref{lowerestimate} and the property in \eqref{RNderivatives}. For every sufficiently small $\gamma \in (0, 1) \subset \R,$ we have $\ds \gamma^{n\ell} \ll_{n, \ell} \, \mu_{X}\left( \left\lbrace \Lambda \in X : \beta_\ell(\Lambda) \leq \gamma \right\rbrace \right).$ It follows that for every sufficiently small $\varepsilon \in (0, 1) \subset \R,$ we have \[ \varepsilon^n = \left(\varepsilon^{1/\ell}\right)^{n\ell}  \ll_{n, \ell} \, \mu_{X}\left( \left\lbrace \Lambda \in X : \beta_\ell(\Lambda) \leq \varepsilon^{1/\ell} \right\rbrace \right) \leq \mu_{X}\left( \left\lbrace \Lambda \in X : \left(\prod_{j=1}^\ell \beta_j(\Lambda) \right) \leq \varepsilon \right\rbrace \right).  \]
\end{proof}

\begin{rmk} \rm
Define $\delta : X \to X$ by $\delta(\Lambda) = \Lambda^*,$ where $\Lambda^*$ denotes the lattice dual to $\Lambda$: if $\Lambda = \Z^n g,$ then $\Lambda^* := \Z^n \left(g^{-1}\right)^\dagger.$ We note that $\delta$ is a homeomorphism that is equal to its own inverse. Since the Haar measure $\mu_G$ is bi-invariant, the inversion map $G \to G$ preserves $\mu_G$; clearly, the transposition map $G \to G$ preserves $\mu_G.$ We conclude that $\delta$ preserves $\mu_X.$ 
\end{rmk} Let us now record as a lemma the following important transference theorem. \begin{lem}\cite[VIII.5, Theorem VI]{Cassels}\label{transference}
There exists $\theta_n \in \R_{>1}$ such that for any integer $j$ with $1 \leq j \leq n$ and any $\Lambda \in X$, we have $\ds 1 \leq \beta_j\left(\Lambda\right) \beta_{n-j+1}\left(\Lambda^*\right) \leq \theta_n$.
\end{lem} 

\begin{proof}[Proof of Proposition \ref{dualvolestimate}]
Let $\theta_n \in \R_{>1}$ be as in Lemma \ref{transference}. Let $R \in \R_{>0}.$ We then have 
\begin{align*}
    \mu_{X}\left( \left\lbrace \Lambda \in X : \beta_\ell(\Lambda) \geq R \right\rbrace \right) &\leq  \mu_{X}\left( \left\lbrace \Lambda \in X : \beta_{n-\ell+1}\left(\Lambda^*\right) \leq \theta_n R^{-1} \right\rbrace \right) \\
    &=  \mu_{X}\left( \left\lbrace \Lambda \in X : \beta_{n-\ell+1}\left(\Lambda\right) \leq \theta_n R^{-1} \right\rbrace \right) \\
    &\ll_{n, \ell} \left(\theta_n R^{-1}\right)^{n(n-\ell+1)} \\
    &\ll_{n, \ell} R^{-n(n-\ell+1)}. 
\end{align*} The first inequality follows from the upper bound in Lemma \ref{transference}, the equality follows from the fact that $\delta$ preserves $\mu_X$, and the second inequality follows from the upper bound in Proposition \ref{volestimate}.

Similarly, we also have \begin{align*}
    \mu_{X}\left( \left\lbrace \Lambda \in X : \left(\prod_{j=\ell}^n \beta_j(\Lambda) \right) \geq R \right\rbrace \right) &\leq  \mu_{X}\left( \left\lbrace \Lambda \in X : \left(\prod_{j=\ell}^n \beta_{n-j+1}\left(\Lambda^*\right) \right) \leq \theta_n^{(n-\ell+1)} R^{-1} \right\rbrace \right) \\
    &= \mu_{X}\left( \left\lbrace \Lambda \in X : \left(\prod_{j=1}^{n-\ell+1} \beta_{j}\left(\Lambda^*\right) \right) \leq \theta_n^{(n-\ell+1)} R^{-1} \right\rbrace \right) \\
    &= \mu_{X}\left( \left\lbrace \Lambda \in X : \left(\prod_{j=1}^{n-\ell+1} \beta_{j}(\Lambda) \right) \leq \theta_n^{(n-\ell+1)} R^{-1} \right\rbrace \right) \\
    &\ll_{n, \ell} \left(\theta_n^{(n-\ell+1)} R^{-1}\right)^{n} \\
    &\ll_{n, \ell} R^{-n}.       
\end{align*}

Let $\varepsilon_{0, n, n-\ell+1} \in \R_{>0}$ be as in Proposition \ref{volestimate}. Set $\ds M_{n, \ell} := \left(\varepsilon_{0, n, n-\ell+1}\right)^{-1}.$ For any $\ds M \in (M_{n, \ell}, +\infty) \subset \R,$ we have \begin{align*}
    M^{-n(n-\ell+1)} = \left(M^{-1}\right)^{n(n-\ell+1)} 
    &\ll_{n, \ell} \, \mu_{X}\left( \left\lbrace \Lambda \in X : \beta_{n-\ell+1}(\Lambda) \leq M^{-1} \right\rbrace \right) \\
    &\ll_{n, \ell} \, \mu_{X}\left( \left\lbrace \Lambda \in X : \beta_{\ell}\left(\Lambda^*\right) \geq M \right\rbrace \right) \\
    &= \mu_{X}\left( \left\lbrace \Lambda \in X : \beta_{\ell}(\Lambda) \geq M \right\rbrace \right).  
\end{align*}

Similarly, we also have \begin{align*}
    M^{-n} = \left(M^{-1}\right)^{n} &\ll_{n, \ell} \, \mu_{X}\left( \left\lbrace \Lambda \in X : \left(\prod_{j=1}^{n-\ell+1}\beta_{j}(\Lambda) \right) \leq M^{-1} \right\rbrace \right) \\
    &\ll_{n, \ell} \, \mu_{X}\left( \left\lbrace \Lambda \in X : \left(\prod_{j=1}^{n-\ell+1}\beta_{n-j+1}\left(\Lambda^*\right) \right) \geq M \right\rbrace \right) \\
    &= \mu_{X}\left( \left\lbrace \Lambda \in X : \left(\prod_{j=1}^{n-\ell+1}\beta_{n-j+1}(\Lambda) \right) \geq M \right\rbrace \right) \\
    &= \mu_{X}\left( \left\lbrace \Lambda \in X : \left(\prod_{j=\ell}^{n}\beta_{j}(\Lambda) \right) \geq M \right\rbrace \right).            
\end{align*}     
\end{proof}

\begin{proof}[Proof of Proposition \ref{Koechervol}]  
Let $R \in \R_{>0}.$ We have $\ds \left\lbrace \Lambda \in X : \sigma_\ell(\Lambda) \leq R^{-1/(2s)} \right\rbrace \subseteq \left\lbrace \Lambda \in X : \zeta_\ell(\Lambda, s) \geq R \right\rbrace.$ The desired lower bound is now an immediate consequence of Lemma \ref{GenSecondMink} and Proposition \ref{volestimate}.

Let us now establish the upper bound. Using Lemma \ref{GenSecondMink} and Proposition \ref{volestimate}, we have  \begin{align*} &\mu_X\left(\left\lbrace \Lambda \in X : \zeta_\ell(\Lambda, s) \geq R \right\rbrace \right)\\
&\leq \mu_X\left(\left\lbrace \Lambda \in X : \sigma_\ell(\Lambda) \leq R^{-1/(2s)} \right\rbrace\right) + \mu_X\left(\left\lbrace \Lambda \in X : \sigma_\ell(\Lambda) \geq R^{-1/(2s)} \ \ \textup{and} \ \ \zeta_\ell(\Lambda, s) \geq R \right\rbrace\right) \\
&\ll_{n, \ell} R^{-n/(2s)} + \mu_X\left(\left\lbrace \Lambda \in X : \sigma_\ell(\Lambda) \geq R^{-1/(2s)} \ \ \textup{and} \ \ \zeta_\ell(\Lambda, s) \geq R \right\rbrace\right).  
\end{align*} 

\smallskip

For each $\Lambda \in X,$ define $\ds U(\Lambda) := \left\lbrace \Theta \in X_k(\Lambda) : \det(\Theta) \geq R^{-1/(2s)} \right\rbrace.$ Markov's inequality then implies 
\begin{equation*}
\mu_X\left(\left\lbrace \Lambda \in X : \sigma_\ell(\Lambda) \geq R^{-1/(2s)} \ \ \textup{and} \ \ \zeta_\ell(\Lambda, s) \geq R \right\rbrace\right) \leq R^{-1} \int_X \sum_{\Theta \in U(\Lambda)} \left(\det(\Theta)\right)^{-2s} \ d\mu_X(\Lambda).   
\end{equation*} It then follows from \cite[Lemma 5]{Thunder} that\footnote{The remark in Footnote \ref{adelicclassical} applies here as well.} \begin{equation*}
R^{-1} \int_X \sum_{\Theta \in U(\Lambda)} \left(\det(\Theta)\right)^{-2s} \ d\mu_X(\Lambda) \ll_{n, \ell} \, R^{-1} \int_{R^{-1/(2s)}}^{+\infty} t^{-2s} \, t^{n-1} \ dt. 
\end{equation*} Since \begin{equation*}
R^{-1} \int_{R^{-1/(2s)}}^{+\infty} t^{-2s} \, t^{n-1} \ dt = (2s-n)^{-1} \cdot R^{-n/(2s)}, 
\end{equation*} the desired upper bound now follows. \end{proof}

We now prove Proposition \ref{prop:k-dl}.

\begin{proof}[Proof of Proposition \ref{prop:k-dl}] The necessary inequalities concerning the $\mu_X$-measures of various sets are immediate consequences of Propositions \ref{volestimate}, \ref{dualvolestimate}, and \ref{Koechervol}; it thus suffices to prove the uniform continuity of the functions in question. Let $\ell$ be any integer with $1 \leq \ell \leq n.$ Let $\varepsilon \in \R_{>0}$ be given. Set $C := \exp(\varepsilon) > 1.$ 

\smallskip

Recall from Notation \ref{notation} that for any $r \in \R_{\geq 0},$ we let $B_r \subset \R^n$ denote the closed Euclidean ball that is centered at the origin and has radius equal to $r.$ For any $g \in G,$ let $\eta(g)$ denote the operator norm of $g$ when both its domain and codomain are equipped with the Euclidean norm on $\R^n$ $\|\cdot\| = \|\cdot\|_1.$ Set $\ds U := \left\lbrace g \in G : \max\left\lbrace \eta(g), \eta\left(g^{-1}\right) \right\rbrace < C \right\rbrace.$ Then $U$ is an open subset of $G$, $\mathrm{id}_G \in U$, and $U = U^{-1}.$ For any $g \in U$ and any $v \in \R^n$, we have \[ C^{-1}\|v\| \leq \|vg\| \leq C\|v\|. \] Let $\Lambda_0 \in X$ and $g_0 \in U$ be arbitrary. Write $\beta := \beta_\ell\left(\Lambda_0\right) \in \R_{>0}.$ We then have $(\Lambda_0 \cap B_\beta)g_0 \subseteq \Lambda_0 g_0 \cap B_{C\beta}$ and thus have $\beta_\ell
\left(\Lambda_0 g_0\right) \leq C\beta.$ For any $\beta' \in (0, \beta) \subset \R$, we have $\Lambda_0 g_0 \cap B_{C^{-1}\beta'} \subseteq (\Lambda_0 \cap B_{\beta'})g_0.$ We then have $\beta_\ell\left(\Lambda_0 g_0\right) \geq C^{-1}\beta'$ and thus have $\beta_\ell\left(\Lambda_0 g_0\right) \geq C^{-1}\beta$. It thus follows \[ -\varepsilon = -\log C \leq (\log \circ \beta_\ell)\left(\Lambda_0 g_0\right) - (\log \circ \beta_\ell)\left(\Lambda_0\right) \leq \log C = \varepsilon. \] This establishes the uniform continuity of $\log \circ \beta_\ell : X \to \R.$ If $S$ is any nonempty subset of $\{1, \dots , n \},$ then it follows immediately that the function $\xi_S : X \to \R$ given by $ \xi_S(\Lambda) := \log \left(  \prod_{j \in S} \beta_j(\Lambda)  \right) = \sum_{j \in S}\log \left( \beta_j(\Lambda) \right)$ is uniformly continuous.    

\smallskip

Given any $g \in G,$ we define $\langle - , - \rangle_{\ell, g}$ to be the inner product on $\bigwedge^\ell\left(\R^n\right)$ given by \[ \langle v_1 \wedge \dots \wedge v_\ell, w_1 \wedge \dots \wedge w_\ell \rangle_{\ell, g} := \langle v_1 g \wedge \dots \wedge v_\ell g, w_1 g \wedge \dots \wedge w_\ell g \rangle_\ell; \] we then let $\|\cdot\|_{\ell, g}$ denote the corresponding norm on $\bigwedge^\ell\left(\R^n\right).$ Define $\kappa : G \to \R_{>0}$ by 
\begin{equation}\label{sortofopnorm}
\kappa(g) := \inf \left\lbrace E \in \R_{>0} : \|\cdot\|_{\ell, g} \leq E \, \|\cdot\|_\ell \right\rbrace.
\end{equation} Observe that the infimum in \eqref{sortofopnorm} is finite because all norms on $\bigwedge^\ell\left(\R^n\right)$ are equivalent; moreover, this infimum is clearly a minimum. We thus conclude that $\kappa$ is well-defined. Let $s \in \left(n/2, +\infty\right) \subset \R$ be given. Set $\ds V := \left\lbrace g \in G : \max\left\lbrace \kappa(g), \kappa\left(g^{-1}\right) \right\rbrace < C^{1/2s} \right\rbrace.$ Then $V$ is an open subset of $G$, $\mathrm{id}_G \in V$, and $V = V^{-1}.$ It now follows that for any $\Lambda \in X$ and any $g \in U,$ we have \[ C^{-1} \zeta_\ell(\Lambda, s) \leq \zeta_\ell(\Lambda g, s) \leq C \, \zeta_\ell(\Lambda, s) \] and thus have \[ -\varepsilon \leq \log\left(\zeta_\ell(\Lambda g, s)\right) - \log\left(\zeta_\ell(\Lambda, s)\right) \leq \varepsilon. \]  We conclude that $\log \circ \zeta_\ell(\cdot, s) : X \to \R$ is uniformly continuous.   \end{proof}

\begin{proof}[Proof of Theorem \ref{loglaw}]
Let $\alpha \in \R_{>0}$ be given, and let $\Delta : X \to \R$ be a function that is $\alpha$-DL; suppose further that for each $g \in \operatorname{SO}(n) \subset G$ and each $\Lambda \in X,$ we have $\Delta(\Lambda) = \Delta(\Lambda g).$ Let $C \in \R_{>1}$ correspond to $\Delta$ as in Definition \ref{DLdefn}. To prove Theorem \ref{loglaw}, it suffices to establish \[ \limsup_{t \to +\infty}  \frac{\Delta\left( \Lambda g_t \right)}{\log t} = \alpha^{-1}. \] We now proceed to do so.   

\smallskip

Let $\gamma \in \R_{>0}$ be given. For each $k \in \Z_{\geq 2},$ define $\ds r_k := \left( \alpha^{-1} + \gamma \right) \log k.$ Since the natural action of $G$ on $X$ preserves the measure $\mu_X$, it follows that for each $k \in \Z_{\geq 2},$ we have \[ \mu_{X}\left( \left\lbrace \Lambda \in X : \Delta\left( \Lambda g_k \right) \geq r_k \right\rbrace \right) = \mu_{X}\left( \left\lbrace \Lambda \in X : \Delta\left( \Lambda \right) \geq r_k \right\rbrace \right) \leq C \, \exp\left(- \alpha r_k\right)  = C \, k^{-(1 + \alpha \gamma)}. 
\] Since $\ds \sum_{k=2}^{+\infty} C \, k^{-(1 + \alpha \gamma)} < +\infty,$ the Borel\textendash Cantelli Lemma implies the following: for $\mu_X$-almost every $\Lambda \in X,$ there exists some $k_\Lambda \in \Z_{\geq 2}$ such that for each $k \in \Z$ with $k \geq k_\Lambda,$ we have $\ds \frac{\Delta\left( \Lambda g_k \right)}{\log k} < \alpha^{-1} + \gamma.$ It follows that for $\mu_X$-almost every $\Lambda \in X,$ we have \[ \limsup_{k \to +\infty}  \frac{\Delta\left( \Lambda g_k \right)}{\log k} \leq \alpha^{-1}. \] A standard continuity argument now implies that for $\mu_X$-almost every $\Lambda \in X,$ we have \[ \limsup_{t \to +\infty}  \frac{\Delta\left( \Lambda g_t \right)}{\log t} \leq \alpha^{-1}. \]

\medskip

For each $z \in \R,$ define $\tau_z : X \to \left( \Z_{\geq 2} \cup \{+\infty\} \right)$ by \[ \tau_z(\Lambda) := \inf\left\lbrace k \in \Z_{\geq 2} : \Delta\left(\Lambda g_k \right)  \geq z \right\rbrace , \] using the convention $\inf \varnothing = +\infty.$ By the Moore Ergodicity Theorem, the natural action of $\left\lbrace g_k \right\rbrace_{k \in \Z}$ on $X$ is ergodic. In particular, for each $z \in \R$ and $\mu_X$-almost every $\Lambda \in X,$ we have $\tau_z(\Lambda) < +\infty.$ For each $g \in \operatorname{SO}(n) \subset G$ and each $z \in \R,$ we note that the image of $\Delta^{-1}\left( [z, +\infty) \right)$ under the natural action of $g$ is equal to $\Delta^{-1}\left( [z, +\infty) \right)$ itself; furthermore, we have \begin{equation}\label{decay}
  \lim_{t \to +\infty}  \frac{-\log\left(\mu_{X}\left( \Delta^{-1}\left( [t, +\infty) \right) \right)\right)}{t} = \alpha.  
\end{equation} It now follows from \cite[Theorem 1.1]{shrinking} that for $\mu_X$-almost every $\Lambda \in X,$ we have \begin{equation}\label{KelmerYuMain}
    \lim_{t \to +\infty} \frac{\log\left(\tau_t(\Lambda) \right)}{-\log\left(\mu_{X}\left( \Delta^{-1}\left( [t, +\infty) \right) \right)\right)} = 1. 
\end{equation}

Now and for the remainder of the proof, fix any $\Lambda \in X$ that satisfies \eqref{KelmerYuMain}. Let $\varepsilon \in \R_{>0}$ be given. Then there exists $T = T_\Lambda \in \R_{>0}$ such that for each $t \in \R_{\geq T},$ we have \begin{equation*} \frac{-\log\left(\mu_{X}\left( \Delta^{-1}\left( [t, +\infty) \right) \right)\right)}{t} \leq \alpha +\varepsilon
\end{equation*} and \begin{equation*}\frac{\log\left(\tau_t(\Lambda) \right)}{-\log\left(\mu_{X}\left( \Delta^{-1}\left( [t, +\infty) \right) \right)\right)} \leq 1 + \varepsilon. 
\end{equation*}

For each $t \in \R_{\geq T},$ we then have \[ 0 < \log 2 \leq \log\left(\tau_t(\Lambda) \right) \leq (\alpha + \varepsilon) \, (1 + \varepsilon) \, t \leq (\alpha + \varepsilon) \, (1 + \varepsilon) \, \Delta\left(\Lambda g_{\tau_t(\Lambda)}\right) < +\infty\] and thus have \begin{equation}\label{proofcomplete} \frac{\Delta\left(\Lambda g_{\tau_t(\Lambda)}\right)}{\log\left(\tau_t(\Lambda) \right)} \geq \frac{1}{(\alpha + \varepsilon) \, (1 + \varepsilon)}. \end{equation} It follows from \eqref{decay} and \eqref{KelmerYuMain} that $\ds   \lim_{t \to +\infty} \frac{\log\left(\tau_t(\Lambda) \right)}{t} = \alpha \in \R_{>0}$; in particular, the sequence $\ds \left( \tau_k(\Lambda) \right)_{k \in \Z_{\geq 1}}$ is unbounded. For each $k \in \Z_{\geq 1},$ set $w_k := \tau_k(\Lambda).$ Let $\ds \left(w_{k_j}\right)_{j \in \Z_{\geq 1}}$ be a subsequence of $\ds \left( w_k \right)_{k \in \Z_{\geq 1}}$ that is strictly increasing and that satisfies $\ds \lim_{j \to +\infty} w_{k_j} = +\infty.$ It then follows from \eqref{proofcomplete} that \[ \limsup_{t \to +\infty}  \frac{\Delta\left( \Lambda g_t \right)}{\log t} \geq \limsup_{j \to +\infty} \frac{\Delta\left(\Lambda g_{w_{k_j}}\right)}{\log\left(w_{k_j} \right)} \geq \frac{1}{(\alpha + \varepsilon) \, (1 + \varepsilon)}. \] Since $\varepsilon \in \R_{>0}$ is arbitrary, it follows that $\ds \limsup_{t \to +\infty}  \frac{\Delta\left( \Lambda g_t \right)}{\log t} \geq \alpha^{-1}.$  \end{proof}

\begin{rmk} \rm  
Let $\ell$ be any integer with $1 \leq \ell \leq n-1.$ It is an immediate consequence of Lemma \ref{GenSecondMink} that Proposition \ref{volestimate} holds when $\prod_{j=1}^\ell \beta_j$ is replaced by $\sigma_\ell.$  Arguing as in the proof of Proposition \ref{prop:k-dl} to establish uniform continuity, it is then easy to see that the function $-\left(\log \circ \, \sigma_\ell\right)$ is $n$-DL. It follows that Theorem \ref{loglaw}\,(i) holds when $\prod_{j=1}^\ell \beta_j$ is replaced by $\sigma_\ell.$ It likewise follows that one may take the function $\Delta$ in Theorem \ref{thm:b-c} to be $\sigma_\ell.$ The reader who is interested in sublattices is encouraged to see the previously cited paper by J.\,L.~Thunder \cite{Thunder} and a related paper by the first-named author \cite{KimSublattices}.     
\end{rmk}  

\section*{Acknowledgements} \noindent  The second-named author would like to thank Dmitry Kleinbock for various discussions; he would also like to thank Jayadev Athreya, Dubi Kelmer, and Kevin O'Neill. The authors are deeply grateful to the anonymous referee for a thorough report that resulted in significant improvements, including the consideration of the Koecher zeta functions.

\end{document}